\documentclass{amsart}

\usepackage{amssymb,color}
\usepackage{amsfonts}
\usepackage{amsmath}
\usepackage{euscript}
\usepackage{enumerate}
\usepackage{pdfsync}
\newtheorem{theorem}{Theorem}[section]
\newtheorem{lemma}[theorem]{Lemma}
\newtheorem{note}[theorem]{Note}
\newtheorem{prop}[theorem]{Proposition}
\newtheorem{cor}[theorem]{Corollary}

\newtheorem*{Theorem1'}{Theorem 1'}

\theoremstyle{definition}

\theoremstyle{remark}

\newcommand \Z{{\mathbb Z}}
\newcommand \Q{{\mathbb Q}}
\newcommand \N{{\mathbb N}}

\begin{document}

\title[A family of groups extending McLain's]{A family of groups extending McLain's}

\author{Leandro Cagliero}
\address{CIEM-CONICET, FAMAF-Universidad Nacional de C\'ordoba, Argentina, and 
Guangdong Technion - Israel Institute of Technology (GTIIT)}
\email{cagliero@famaf.unc.edu.ar}
\thanks{The first author was partially supported by grants CONICET PIP 11220210100597CO and  SeCyT-UNC\'ordoba 33620230100117CB}


\author{Fernando Szechtman}
\address{Department of Mathematics and Statistics, University of Regina, Canada}
\email{fernando.szechtman@gmail.com}
\thanks{The second author was partially supported by NSERC grant 2020-04062}



\begin{abstract} Given a strict partial order $\Delta$ on a set $\Lambda$ and an arbitrary ring $R$ with $1\neq 0$,
the corresponding McLain group $M(\Delta)$ has been studied in depth. We construct a larger family of 
McLain groups $G(\Delta)$, where $\Delta$ need be neither asymmetric nor transitive, while satisfying two weaker axioms.
Structural properties common to all members~$G(\Delta)$ of this new family are investigated, including
a group presentation, a description of the factors of its descending central series,
a canonical form for its elements relative to any total order on~$\Delta$, and a recursive 
determination of its upper central series. In addition, we prove the natural isomorphism
$G(\Delta)/G(\Gamma)\cong G(\Delta\setminus\Gamma)$,
where $\Gamma$ is a normal subset of $\Delta$, and $G(\Gamma)$ and $G(\Delta\setminus\Gamma)$
are extended McLain groups on their own right. This result has no parallel in the classical context.
\end{abstract}

\maketitle

\section{Introduction}

In a two page long note published in 1954, McLain \cite{M} constructed an infinite characteristically simple 
locally finite $p$-group $M(F_p,<)$,
where $p$ a prime and $<$ is the usual order on $\Q$, showing that $M(F_p,<)$ 
satisfies several interesting properties at the same time, making his paper a widely cited source ever since.

It eventually become clear that McLain's construction (albeit not his conclusions) 
went through in ample generality, thus producing a McLain group $M(R,\Delta)$ associated to 
any ring $R$ with $1\neq 0$ and any set $\Lambda$ endowed with a strict partial order $\Delta$.
Droste and G\"{o}bel  \cite{DG} studied $M(R,\Delta)$ in this context. Under 
suitable conditions, they: show that $R$ and $\Delta$ can be recovered from $M(R,\Delta)$;  
determine the automorphism group of $M(R,\Delta)$, extending prior work by Roseblade~\cite{Ro}; 
analyze properties of $M(R,\Delta)$ that ultimately depend on $R$ and $\Delta$ (e.g. being characteristically
simple); and use  $M(R,\Delta)$ to produce interesting applications, one of which supersedes
previous examples given by Wilson \cite{W}.

Further group theoretical properties of McLain's groups were studied, among others, by Berrick and Downey~\cite{BD},
Hołubowski \cite{H}, Cutolo and Smith \cite{CS},
and Arikan \cite{Ar}.

In a totally different direction, Szechtman, Herman, and Izadi \cite{SHI} study 
the representation theory of $M(R,\Delta)$ in arbitrary characteristic,
extending to this scenario features of the unitriangular group $U_n(q)$, $q=p^m$, developed by Andr\'e \cite{A, A2, A3} over
the complex numbers. Under mild assumptions, \cite{SHI} constructs possibly infinite dimensional irreducible elementary
modules, and proceeds to investigate certain $\Delta$-allowed tensor products of them called basic. It turns out
that all basic modules are monomial, distinct basic modules are disjoint, and nested basic modules are themselves irreducible. Moreover, exact conditions for a basic module to be irreducible  are given, together with
a description of how a basic module decomposes as the direct sum of irreducible modules, including a criterion
for this decomposition to be multiplicity free.

In the present paper, we construct and study a further extension of McLain groups. To build these groups
we require, as above, an arbitrary ring $R$ with $1\neq 0$, a set $\Lambda$, and a relation $\Delta$
on $\Lambda$. However, we subject $\Delta$ to two weak axioms which
are automatically satisfied by every strict partial order and not the other way around.
Associated to this data we produce a McLain group $G(\Delta)$ and proceed to investigate some its structural features.
We first prove the natural isomorphism 
$$
G(\Delta)/G(\Gamma)\cong G(\Delta\setminus\Gamma),
$$
where $\Gamma$ is a normal subset $\Gamma$ of $\Delta$, and $G(\Gamma)$ and $G(\Delta\setminus\Gamma)$
are extended McLain groups on their own right. There is no analogue for result in the classical context,
as $\Delta\setminus\Gamma$ fails in general to be transitive when $\Delta$ is a strict partial order
and $\Gamma$ is normal in $\Delta$. We then give a group presentation for $G(\Delta)$, a description of the factors of its descending central series,
a canonical form for its elements with respect
to any total order on $\Delta$, and we recursively determine the upper central series of~$G(\Delta)$.

A simple example of a family of groups $G(\Delta)$ that do not exist in the classical sense is the following.
Given $n\geq 4$, take $\Lambda_n=\Z_n$ and
$$
\Delta_n=\{(i,i+1)\,|\, i\in\Z_n\}\cup \{(i,i+2)\,|\, i\in\Z_n\}.
$$
Thus $\Delta_n$ is an oriented $n$-gon with an oriented diagonal between vertices that are two steps apart.
Axioms (A1)-(A2) are readily verified, and it is clear that $\Delta_n$ is not even remotely close to being a strict partial order.
A noteworthy property of $G(\Delta_n)$ is that a typical factorization property of the classical McLain groups completely fails
in this case. See Note \ref{eje} below for details.

\section{Definition of the extended McLain groups}

We henceforth fix an arbitrary ring $R$ with $1\neq 0$,
a set $\Lambda$ and a subset $\Delta$ of $\Lambda\times \Lambda$ satisfying the following axioms:

\medskip

(A1) $\Delta$ is irreflexive, that is, $\Delta\cap 1_\Lambda=\emptyset$.

\medskip

(A2) If $(i,j),(j,k), (k,\ell), (i,\ell)\in\Delta$, then $(i,k)\in\Delta\Leftrightarrow (j,\ell)\in\Delta$.

\medskip

Let $R\Delta$ be a free $R$-module with basis $e_{i,j}$, where $(i,j)\in\Delta$. We make $R\Delta$ into a ring by declaring
$$
e_{i,j}e_{k,\ell}=\begin{cases} e_{i,\ell} & \text{ if }j=k\text{ and }(i,\ell)\in\Delta,\\ 0 & \text{ otherwise.}
\end{cases}
$$

\begin{lemma}\label{distinct} Suppose $n\geq 3$ and $i_1,\dots,i_n\in\Lambda$ satisfy:

\noindent $\bullet$ $(i_1,i_k)\in\Delta$ for all $1<k\leq n$,

\noindent $\bullet$ $(i_k,i_{k+1})\in\Delta$ for all $1\leq k<n$.

Then $(i_k,i_{j})\in\Delta$ for all $1\leq k<j\leq n$. In particular, $i_1,\dots,i_n$ are all distinct.
\end{lemma}

\begin{proof} This follows easily by induction on $n$ by means of (A1) and (A2).
\end{proof}

\begin{lemma}\label{nil} Every element of $R\Delta$ is nilpotent.
\end{lemma}

\begin{proof} This follows immediately from Lemma \ref{distinct}.
\end{proof}

We adjoin 1 to $R\Delta$ and produce the ring $R\cdot 1\oplus R\Delta$, whose unit group contains the extended McLain group
associated to $\Delta$, namely
$$
G(\Delta)=\{1+x\,|\, x\in R\Delta\}.
$$

For $(i,j)\in\Delta$ and $a\in R$, we set 
$$
x_{i,j}(a)=1+ae_{i,j},
$$
and write $[u,v]=uvu^{-1}v^{-1}$ for all $u,v$ belonging to a given group $H$. Then 
\begin{equation}\label{rel1}
[x_{i,j}(a),x_{j,k}(b)]
=\begin{cases} x_{i,k}(ab) & \text{ if }(i,k)\in\Delta,\\ 1 & \text{ otherwise,}
\end{cases}
\end{equation}
\begin{equation}\label{rel2}
[x_{i,j}(a),x_{k,\ell}(b)]
=1\text{ if }j\neq k\text{ and }i\neq\ell,
\end{equation}
\begin{equation}\label{rel3}
x_{i,j}(a)x_{i,j}(b)
=x_{i,j}(a+b).
\end{equation}

Given a subset $\Gamma$ of $\Delta$, we write $R\Gamma$ for the $R$-span of all $e_{i,j}$, $(i,j)\in\Gamma$, 
and define
$$
G(\Gamma)=\{1+x\,|\, x\in R\Gamma\},
$$
which is a subgroup of $G(\Delta)$ if and only if $\Gamma$ is closed, in the sense that
$$
(i,j),(j,k)\in\Gamma, (i,k)\in\Delta\Rightarrow (i,k)\in\Gamma.
$$

If $\Gamma$ is closed, then $G(\Gamma)$ is the extended McLain group associated to $\Gamma$, as seen below.

\begin{lemma}\label{closed} Suppose $\Gamma$ is a closed subset of $\Delta$.  Then $\Gamma$ satisfies (A1) and (A2).
\end{lemma}

\begin{proof} This is a routine verification.
\end{proof}

Let $\Gamma$ be a subset of $\Delta$. We say that $\Gamma$ is normal if
$$
(i,j)\in\Gamma, (j,k),(i,k)\in\Delta\Rightarrow (i,k)\in\Gamma,
$$
and
$$
(i,j)\in\Gamma, (k,i),(k,j)\in\Delta\Rightarrow (k,j)\in\Gamma,
$$
in which case $\Gamma$ is obviously closed; moreover, $G(\Gamma)$ is a normal subgroup of $G(\Delta)$
and $\Delta\setminus\Gamma$ satisfies (A1)-(A2), as seen below.

\section{An isomorphism between $G(\Delta)/G(\Gamma)$ and $G(\Delta\setminus \Gamma)$}

Associated to every element $x$ of $R\Delta$ there exists a unique finite subset $\Omega(x)$ of $\Delta$
and unique $0\neq r_{i,j}\in R$ for each $(i,j)\in\Omega(x)$ such that
\begin{equation}\label{normalf}
x=\underset{(i,j)\in\Omega(x)}\sum r_{i,j} e_{i,j}.
\end{equation}

Given subsets $\Gamma_1,\Gamma_2$ of $\Delta$, we define $[\Gamma_1,\Gamma_2]$ to be the set 
of all $(i,k)\in\Delta$ such that there is $j\in\Lambda$ such that $(i,j)\in\Gamma_1$ and
$(j,k)\in\Gamma_2$, or there is $j\in\Lambda$ such that $(i,j)\in\Gamma_2$ and
$(j,k)\in\Gamma_1$. 

For a subset $\Gamma$ of $\Delta$, we let $\gamma_1(\Gamma)=\Gamma$, $\gamma_2(\Gamma)=[\Gamma,\Gamma]$,
$\gamma_3(\Gamma)=[\gamma_2(\Gamma),\Gamma]$, etc. Given a group~$H$, we write $\gamma_1(H)=H$, $\gamma_2(H)=[H,H]$,
$\gamma_3(H)=[\gamma_2(H),H]$, etc. Regarding these groups, we will repeatedly and implicitly use generating
properties of them, as described in \cite[Proposition 5.1.7]{R}.

\begin{lemma}\label{nilp} Suppose $\Gamma$ is a finite closed subset of $\Delta$, with $|\Gamma|=n\in\N, n>1$. Then
$\gamma_n(\Gamma)=\emptyset$.
\end{lemma}

\begin{proof} This is an immediate consequence of Lemma \ref{distinct}.
\end{proof}

\begin{lemma}\label{nor} Let $\Gamma_1,\Gamma_2$ be normal subsets of $\Delta$. Then $[\Gamma_1,\Gamma_2]$ 
is a normal subset of $\Delta$ contained in $\Gamma_1\cap\Gamma_2$. In particular, if $\Gamma$ is a closed subset
of $\Delta$, then all $\gamma_k(\Gamma)$ are normal subsets of $\Gamma$.
\end{lemma}

\begin{proof} It is obvious that $[\Gamma_1,\Gamma_2]$ is contained in $\Gamma_1\cap\Gamma_2$. That 
that $[\Gamma_1,\Gamma_2]$ is
normal in $\Delta$ is a routine verification. The last assertion now follows from Lemma \ref{closed}.
\end{proof}

We write $\pi_1,\pi_2:\Lambda\times\Lambda\to\Lambda$ for the canonical projections $(i,j)\mapsto i$ and
$(i,j)\mapsto j$, respectively. Given any finite subset $\Omega$ of $\Delta$,
we set $\Phi(\Omega)=\pi_1(\Omega)\cup\pi_2(\Omega)$, noting that $(\Phi(\Omega)\times \Phi(\Omega))\cap \Delta$
is a finite closed subset of $\Delta$ containing $\Omega$, and we define 
$\Gamma(\Omega)$ as the intersection of all finite closed subsets of $\Delta$ containing $\Omega$.
In addition, we let $T(\Delta)$ stand for the subgroup of $G(\Delta)$ generated by all  
$x_{i,j}(a)$ with $a\in R$ and $(i,j)\in\Delta$.

\begin{prop}\label{product} We have $G(\Delta)=T(\Delta)$.
\end{prop}

\begin{proof} Let $g\in G(\Delta)$. Then $g=1+x$, where $x\in R\Delta$ is as in (\ref{normalf}).
We show that $g\in T(\Delta)$ by induction
on the size of $\Gamma=\Gamma(\Omega(x))$. The result is clear if $\Gamma=\emptyset$. Suppose that $\Gamma\neq\emptyset$
and that $h\in T(\Delta)$ when $h=1+y$, $y\in R\Delta$, and $|\Gamma(\Omega(y))|<|\Gamma|$.
Since $[\Gamma,\Gamma]$ is closed by Lemma \ref{nor} and is properly included in $\Gamma$ by Lemma \ref{nilp}, 
we see that $\Omega(x)$ is not included in $[\Gamma,\Gamma]$.
Multiplying $g$ on the left by the product in any order of all $1-r_{i,j}e_{i,j}$, where $r_{i,j} e_{i,j}$ appears 
in~(\ref{normalf}) and $(i,j)\in\Omega(x)\setminus [\Gamma,\Gamma]$, we obtain an element $h=1+y$
such that $\Omega(y)$ is included in $[\Gamma,\Gamma]$, so $\Gamma(\Omega(y))\subseteq [\Gamma,\Gamma]$.
By inductive hypothesis
$h\in T(\Delta)$, and hence $g\in T(\Delta)$.
\end{proof}

\begin{note}\label{eje} It is not true in general that given $g\in G(\Delta)$, $g=1+x$, with $x$ as in (\ref{normalf}), then
$g$ is the product of the factors $1+r_{ij}e_{ij}$ in some order, a property considered crucial in \cite{DG}.
Thus, the proofs of \cite[Lemma 1]{M} and \cite[Lemma 2.1]{DG} break down in our context.
Indeed, given $n\geq 4$, take $\Lambda_n=\Z_n$ and
$$
\Delta_n=\{(i,i+1)\,|\, i\in\Z_n\}\cup \{(i,i+2)\,|\, i\in\Z_n\}.
$$
Axioms (A1)-(A2) are readily verified.
Then $g=1+\underset{i\in\Z_n}\sum e_{i,i+1}$ is not equal to the product $h=\underset{i\in\Z_n}\Pi x_{i,i+1}(a_i)$
in any order for any choice of $a_i\in R$, as $h$ is equal to $1+x+y$, where $x=\underset{i\in\Z_n}\sum a_i e_{i,i+1}$
and $y$ is the sum of all $a_i a_{i+1} e_{i,i+2}$ such that $x_{i,i+1}(a_i)$ appears before $x_{i+1,i+2}(a_{i+1})$
in $h$.
 
Given a subset $\Omega$ of $\Delta$, we say that $(i,j)\in\Omega$ is maximal (resp. minimal) if there is no $(j,k)\in\Omega$ 
(resp. $(k,i)\in\Omega$) such that $(i,k)\in\Delta$
(resp. $(k,j)\in\Delta$). It is easy to see that if every nonempty finite subset of $\Delta$ has
a maximal (resp. minimal) element, and $g\in G(\Delta)$, $g=1+x$ with $x$ as in~(\ref{normalf}), then
$g$ is the product of the factors $1+r_{ij}e_{ij}$ in some order. Note that in the above examples
$\Omega_n=\{(i,i+1)\,|\, i\in\Z_n\}$ does not have a maximal (resp. minimal) element.
\end{note}

\begin{prop}\label{normal} Let $\Gamma$ be a normal subset of $\Delta$. Then $G(\Gamma)$
is a normal subgroup of $G(\Delta)$.
\end{prop}

\begin{proof} This follows from Proposition \ref{product} and the relations (\ref{rel1})-(\ref{rel3}).
\end{proof}

\begin{lemma}\label{axiom} Let $\Gamma$ be a normal subset of $\Delta$. Then $\Delta\setminus \Gamma$ satisfies (A1) and (A2).
\end{lemma}

\begin{proof} This is a routine verification.
\end{proof}

\begin{theorem}\label{quotient} Let $\Gamma$ be a normal subset of $\Delta$. Then 
$
G(\Delta)/G(\Gamma)\cong G(\Delta\setminus \Gamma).
$
\end{theorem}

\begin{proof} The natural projection $R\Delta\to R(\Delta\setminus \Gamma)$ is
a ring epimorphism, yielding a group epimorphism $G(\Delta)\to G(\Delta\setminus \Gamma)$
with kernel $G(\Gamma)$.
\end{proof}

\begin{cor} Let $\Gamma$ be a normal subset of $\Delta$. Then the subset
$
G(\Delta\setminus \Gamma)
$
of $G(\Delta)$ is a set of representatives for the cosets of $G(\Gamma)$ in $G(\Delta)$.
\end{cor}

\section{A group presentation for $G(\Delta)$}\label{t}

Let $H(\Delta)$ be the abstract group generated by elements $y_{i,j}(a)$, with  $(i,j)\in\Delta$ and $a\in R$, subject
to the defining relations (\ref{rel1})-(\ref{rel3}) (where, of course, each appearance of $x$ is to be substituted by~$y$).
For each closed subset $\Gamma$ of $\Delta$, let 
$H(\Gamma)$ be the subgroup of $H(\Delta)$ generated by all 
$y_{i,j}(a)$, with  $(i,j)\in\Gamma$ and $a\in R$.

\begin{prop}\label{gak} For $k\geq 1$ and each closed subset $\Gamma$ of $\Delta$, we have $\gamma_k(H(\Gamma))=H(\gamma_k(\Gamma))$.
\end{prop}

\begin{proof} This follows from Lemma \ref{nor} and the defining relations of $H(\Delta)$.
\end{proof}

\begin{theorem}\label{dir} For $k\geq 1$ and each closed subset $\Gamma$ of $\Delta$, we have 
$$\gamma_k(G(\Gamma))/\gamma_{k+1}(G(\Gamma))\cong (R^+)^{(\gamma_k(\Gamma)\setminus \gamma_{k+1}(\Gamma))},
$$
the restricted direct product of $|\gamma_k(\Gamma)\setminus \gamma_{k+1}(\Gamma)|$ copies of $R^+$.
\end{theorem}

\begin{proof} By Proposition \ref{product}, we have a group epimorphism $f:H(\Delta)\to G(\Delta)$ sending each
$y_{i,j}(a)$ into $x_{i,j}(a)$ and hence $H(\Gamma)$ onto $G(\Gamma)$. Thus,
$\gamma_k(G(\Gamma))=G(\gamma_k(\Gamma))$ and $\gamma_{k+1}(G(\Gamma))=G(\gamma_{k+1}(\Gamma))$ by Proposition \ref{gak}.
By Lemma \ref{closed}, $G(\gamma_k(\Gamma))$ is the extended McLain group associated to $\gamma_k(\Gamma)$.
All generators of the extended McLain group associated
to $\gamma_k(\Gamma)\setminus \gamma_{k+1}(\Gamma)$ commute with each other, so the result follows from Theorem \ref{quotient}.
\end{proof}

\begin{prop}\label{ln} Given any $n\in\N$ and $h_1,\dots,h_n\in H(\Delta)$, there is a finite and closed subset
$\Gamma$ of $\Delta$ such that $h_1,\dots,h_n\in H(\Gamma)$, which is nilpotent. 
In particular, $H(\Delta)$ is locally nilpotent.
\end{prop}

\begin{proof} Let $h_1,\dots,h_n\in H(\Delta)$. Corresponding to each $h_k$ there is a finite subset $\Omega_k$ of $\Delta$ 
such that $h_k$ is the finite product of elements $y_{i,j}(a)$, with $(i,j)\in\Omega_k$ and $a\in R$. 
Let $\Omega$ be the union of all $\Omega_k$, and set $\Gamma=\Gamma(\Omega)$. Then $\Gamma$ is a finite closed subset of $\Delta$
and each $h_k$ belongs to $H(\Gamma)$.
It follows from Lemma \ref{nilp} and Proposition \ref{gak} that $H(\Gamma)$ is nilpotent.
\end{proof}

\begin{theorem}\label{main} The extended McLain group $G(\Delta)$ is generated by all $x_{i,j}(a)$, with $(i,j)\in\Delta$ and $a\in R$,
subject to the defining relations (\ref{rel1})-(\ref{rel3}).
\end{theorem}

\begin{proof} By Proposition \ref{product}, we have a group epimorphism $f:H(\Delta)\to G(\Delta)$ sending each
$y_{i,j}(a)$ into $x_{i,j}(a)$. Suppose $1\neq h\in H(\Delta)$.
We want to show that $f(h)\neq 1$. By Proposition \ref{ln}, there is a finite and closed subset
$\Gamma$ of $\Delta$ such that $h\in H(\Gamma)$, which is nilpotent. As $h\neq 1$, there is $k\in\N$ such that $h\in\gamma_k(H(\Gamma))$
but $h\notin\gamma_{k+1}(H(\Gamma))$. We deduce from Proposition \ref{gak} and the defining relations of $H(\Delta)$
that there is a finite nonempty subset $\Omega$ of $\gamma_k(\Gamma)\setminus \gamma_{k+1}(\Gamma)$, nonzero elements
$a_{i,j}$ of $R$ for each $(i,j)\in \Omega$, and an element $t$ of $\gamma_{k+1}(H(\Gamma))$, such that
$$h=\underset{(i,j)\in\Omega}\Pi y_{i,j}(a_{i,j})t.$$ Therefore
$f(h)\in \gamma_k(G(\Gamma))=G(\gamma_k(\Gamma))$ is not in $\gamma_{k+1}(G(\Gamma))=G(\gamma_{k+1}(\Gamma))$, so
$f(h)\neq 1$.
\end{proof}

When $R$ is a field and $\Delta$ is a strict total order on $\Lambda$, \cite{BD} gives a simpler proof of Theorem \ref{main},
appealing to the fact that each $e_{ij}$ is the endomorphism of an $R$-vector space with basis $(v_{i})_{i\in\Lambda}$
that sends $v_i$ to $v_j$ and annihilates every other basis vector. This argument breaks down in our context, where $e_{ij}e_{jk}$
need not equal $e_{ik}$, so the $e_{ij}$ cannot be viewed as such endomorphisms. Moreover, their proof uses order properties that
do not hold in our setting, as indicated in Note~\ref{eje}.

\section{A canonical form for the elements of $G(\Delta)$}

Let $g\in G(\Delta)$. Then, as seen in Section \ref{t}, there is a finite and closed subset
$\Gamma$ of $\Delta$ such that $g\in G(\Gamma)$, with $G(\Gamma)$ is nilpotent. The intersection, say $\Gamma(g)$, of all such subsets 
$\Gamma$ also has this property and is uniquely determined by $g$.

Let $n\geq 0$ be the nilpotency class of $G(\Gamma)$. Then, there exist unique elements $g_1,\dots,g_{n-1}$
such that each $g_k$ belongs to the subset $G(\gamma_k(\Gamma)\setminus \gamma_{k+1}(\Gamma))$
of $G(\Gamma)$, and $g=g_1\cdots g_{n-1}$. Note that if $g=1$ this is vacuously true, and that if we take $\Gamma=\Gamma(g)$
and $g\neq 1$, then $g_1\neq 1$. Each $g_k$ is the finite product of unique elements 
$x_{i,j}(a_{i,j})$ with $(i,j)\in \gamma_k(\Gamma)\setminus \gamma_{k+1}(\Gamma)$  and $a_{i,j}\in R$. 
We may impose a strict total order $<$
on $\Gamma$ such that if $(i,j)\in \gamma_k(\Gamma)$ and $(u,v)\in \gamma_\ell(\Gamma)$ with $k$ less than $\ell$, 
then $(i,j)<(u,v)$.

For $(i,j)\in\Delta$, set $X_{i,j}=\{x_{i,j}(a)\,|\, a\in R\}$, the basic subgroup of $G(\Delta)$
associated to $(i,j)$.

\begin{prop}\label{canon} Let $\Gamma$ be a finite and closed subset
of $\Delta$ and let $\prec$ be any strict order on $\Gamma$.
Then $G(\Gamma)$ is equal to the product, in increasing $\prec$-order,
of its subgroups $X_{i,j}$, $(i,j)\in\Gamma$, with uniqueness of expression.
\end{prop}

\begin{proof} The above discussion settles the result for $<$, whose choice ensures that 
given any $(i,j)\in\Gamma$, the product, in increasing $<$-order, of all $X_{u,v}$ such that
$(u,v)\in\Gamma$ and $(i,j)\leq (u,v)$, is a normal subgroup of $G(\Gamma)$. The result now follows
from Proposition \ref{gt} below.
\end{proof}

\begin{prop}\label{gt} Let $G$ be a group with subgroups $G_1,\dots,G_n$ such that $G=G_1\cdots G_n$
with uniqueness of expression, and satisfying $G_i\cdots G_n\trianglelefteq G$ for all $1\leq i\leq n$. Then
for any $\sigma\in S_n$, we have $G=G_{\sigma(1)}\cdots G_{\sigma(n)}$, with uniqueness of expression.
\end{prop}

\begin{proof} By induction on $n$. If $n=1$ there is nothing to do. Suppose $n>1$ and the result is true for $n-1$.
By hypothesis, $G_n\trianglelefteq G$. Let $\pi:G\to G/G_n$ be the canonical projection, and set $H=\pi(G)$, $H_i=\pi(G_i)$
for all $1\leq i\leq n$. Let $\sigma\in S_n$. We may assume without loss that $\sigma(n)=n$ in both the existence and the uniqueness of the decomposition at stake. By inductive hypothesis, $H=H_{\sigma(1)}\cdots H_{\sigma(n-1)}$, with uniqueness of expression.
Thus $G=G_{\sigma(1)}\cdots G_{\sigma(n)}$.
Suppose $f_{\sigma(1)}\cdots f_{\sigma(n)}=g_{\sigma(1)}\cdots g_{\sigma(n)}$, 
where $f_{\sigma(i)}, g_{\sigma(i)}\in G_{\sigma(i)}$
for all $1\leq i\leq n$. Applying $\pi$ to both sides and appealing to uniqueness in $H$, 
we infer the existence of $e_{\sigma(1)},\dots,e_{\sigma(n-1)}\in G_n$ such that $f_{\sigma(i)}=g_{\sigma(i)}e_{\sigma(i)}$
for all $1\leq i\leq n-1$. Uniqueness in $G$ forces  $f_{\sigma(i)}=g_{\sigma(i)}$ for all 
$1\leq i\leq n-1$, which then yields $f_{n}=g_n$.
\end{proof}

\begin{theorem}\label{dec} Let $\prec$ be any strict order on $\Delta$. Then $G(\Delta)$
is equal to the product, in increasing $\prec$-order,
of its subgroups $X_{i,j}$, $(i,j)\in\Delta$, with uniqueness of expression.
\end{theorem}

\begin{proof} Given any $g\in G(\Delta)$, there is a finite and closed subset $\Gamma$ of $\Delta$
such that $g\in G(\Gamma)$,
so existence follows from Proposition \ref{canon}. Any 2 decompositions for $g$ will occur in some common $G(\Gamma)$, with 
$\Gamma\subseteq \Delta$ finite and closed, and we have uniqueness there, also by Proposition \ref{canon}.
\end{proof}

\section{Upper central series of $G(\Delta)$}

We say that $(i,j)\in\Delta$ is isolated if there is no $(j,k)\in \Delta$ such that $(i,k)\in\Delta$, and
there is no $(\ell,i)\in \Delta$ such that $(\ell,j)\in\Delta$.
Let $\Gamma_1$ be the subset of $\Delta$ consisting of all its isolated elements.  It is clear  
that $\Gamma_1$ is normal.

\begin{prop}\label{center} The center of $G(\Delta)$ is equal to $G(\Gamma_1)$.
\end{prop}

\begin{proof} Clearly $G(\Gamma_1)$ is contained in the center $G(\Delta)$. Suppose that for some finite $\Omega\subseteq\Delta$,
$$
g=1+\underset{(i,j)\in \Omega}\sum a_{i,j} e_{i,j}
$$
is in the center of $G(\Delta)$. Assume $(i,j)\in \Omega$ is not isolated. 
Then either there is a $(j,k)\in\Delta$ such that $(i,k)\in\Delta$, or there is $(\ell,i)\in \Delta$ such that $(\ell,j)\in\Delta$.
In the first case, when writing $g(1+e_{j,k})$ and $(1+e_{j,k})g$
as $1$ plus an $R$-linear combination of the elements $e_{u,v}$
with $(u,v)\in\Delta$, the coefficients of $e_{i,k}$ are equal to $a_{i,k}+a_{i,j}$ and 
$a_{i,k}$, respectively. Thus $a_{i,j}=0$.
In the second case, when writing $(1+e_{\ell,i})g$ and $g(1+e_{\ell,i})$
as $1$ plus an $R$-linear combination of the elements $e_{u,v}$
with $(u,v)\in \Delta$, the coefficients of $e_{\ell,j}$ are equal to $a_{\ell,j}+a_{i,j}$ and
$a_{\ell,j}$, respectively. Thus~$a_{i,j}=0$.
\end{proof}

Imposing a well-order on $\Delta$, we may now use Lemma \ref{axiom}, Theorem  \ref{quotient}, and Proposition \ref{center}
(together with the fact that the union of an ascending chain of normal subsets of $\Delta$ is normal), to recursively produce the ascending central
series of $G(\Delta)$ (which may be trivial, stabilize at an intermediate point, or reach $G(\Delta)$).



\end{document}